\documentclass[a4paper,leqno,11pt,twoside]{amsart}
\usepackage{amsmath,amsfonts,amssymb,amsthm,amscd}
\usepackage[utf8]{inputenc}
\usepackage[english]{babel}
\usepackage[colorlinks=true,citecolor=blue, urlcolor=blue, linkcolor=blue,pagebackref]{hyperref}
\usepackage[top=1.1in,bottom=1.1in,left=1.in,right=1.in]{geometry} 
\usepackage{graphicx}
\usepackage{wrapfig}
\usepackage{paralist}
\usepackage{tabto}
\usepackage{standalone}
\usepackage{xfrac}
\usepackage{float}
\usepackage{tikz-cd}

\usepackage{pgf,tikz}
\usetikzlibrary{arrows,chains,
 decorations.pathreplacing,
shapes,%
}

\usepackage[all]{xy} 
\usepackage{enumerate}
\usepackage{xcolor}
\usepackage{aliascnt}

\theoremstyle{plain}

\newtheorem{thmIntr}{Theorem}

\newaliascnt{propIntr}{thmIntr}
\newtheorem{propIntr}[propIntr]{Proposition}

\newaliascnt{corIntr}{thmIntr}

\newaliascnt{lem}{thm}
\newtheorem{lem}[lem]{Lemma}
\aliascntresetthe{lem}

\newaliascnt{cor}{thm}
\newtheorem{cor}[cor]{Corollary}
\aliascntresetthe{cor}

\newaliascnt{prop}{thm}
\newtheorem{prop}[prop]{Proposition}
\aliascntresetthe{prop}

\theoremstyle{definition}

\newaliascnt{rem}{thm}

\aliascntresetthe{rem}

\newaliascnt{defn}{thm}

\aliascntresetthe{defn}

\newaliascnt{ex}{thm}
\newtheorem{ex}[ex]{Example}
\aliascntresetthe{ex}

\newaliascnt{setup}{thm}
\newtheorem{setup}[setup]{Set-up}
\aliascntresetthe{setup}

\numberwithin{equation}{section}

\definecolor{applegreen}{rgb}{0.55, 0.71, 0.0}
\newcommand{\commg}[1]{{\color{applegreen}#1}} %amb un altre color, per si hi ha diàleg

\newcommand{\covgen}{\textnormal{cov.gen}}

\newcommand{\Sym}{\textnormal{Sym}}

\title[Curves on the product of two $K$-trivial surfaces]{Curves on the product of two $K$-trivial surfaces}
\author[F.~Moretti and G.~Passeri ]{Federico Moretti and Giovanni Passeri}

\address{Department of Mathematics, State University of New York at Stony Brook, 100 Nicolls Rd, Stony Brook, NY 11794, USA}
\email{federico.moretti@stonybrook.edu}
\email{giovanni.passeri@stonybrook.edu} 

\begin{document}

\begin{abstract}
We study curves on the product of two $K$-trivial surfaces. In the case of the product of two very general abelian surfaces $A_1\times A_2$, we prove that the minimal genus of a non-trivial curve on $A_1\times A_2$ is $6$.
\end{abstract}

\maketitle
%--------------------------------------

\setcounter{tocdepth}{1}

\section{Introduction}

The study of curves on abelian varieties has long been central to algebraic geometry. For instance, determining the minimal algebraic curve class on a very general abelian variety is a difficult open problem; only recently has it been resolved up to dimension $6$ (cf.\ \cite{engel2025optimalityprymtyurinconstructionmathcala6,engel2026matroidsintegralhodgeconjecture,AlexeevDonagiFarkasIzadiOrtega+2020+163+217}). Another fundamental question concerns the minimal genus of a curve on a very general abelian variety. The best available lower bound in dimension $a>3$ is $\frac{a(a-1)}{2}+1$ (see \cite[Theorem~1]{pietro1995generic}), and this bound is optimal for $a=4,5$. The main result of this note is that the minimal genus of a non-trivial curve on the product of two very general abelian surfaces is $6$. Note that, in view of the bound above, this is smaller than the corresponding lower bound for a very general abelian fourfold (of course, a very general abelian fourfold is not isogenous to a product of abelian surfaces). A more Hodge-theoretic reformulation is as follows: given two very general abelian surfaces $A_1$ and $A_2$, the minimal genus of a Jacobian admitting $A_1\times A_2$ as a quotient is $6$.

Our motivation comes from invariants introduced by Lazarsfeld and Martin to compare the birational geometry of projective varieties via correspondences. Let $X$ and $Y$ be irreducible projective varieties of the same dimension $n$. A correspondence between $X$ and $Y$ is an irreducible, reduced $n$-dimensional subscheme $Z\subset X\times Y$ dominating both factors under the projections. Following \cite{RobOli}, one defines
\[
\mathrm{corr.deg}(X,Y)=\min_{Z\subset X\times Y}\Bigl\{\deg(Z\to X)\cdot \deg(Z\to Y)\Bigr\},
\]
where $Z$ ranges over correspondences. Note that $\mathrm{corr.deg}(X,Y)=1$ if and only if $X$ is birational to $Y$.

A different way to measure the complexity of a variety is through curves. Recall the covering genus
\[
\mathrm{cov.gen}(Z)=\min\Bigl\{\, g \ \Big|\ \exists \text{ an irreducible curve of geometric genus } g \text{ passing through a general point of } Z \Bigr\}.
\]
One may then define
\[
\mathrm{cov.gen}(X,Y)=\min_{Z\subset X\times Y}\mathrm{cov.gen}(Z),
\]
where, again, $Z$ ranges over correspondences. If $X$ and $Y$ are birational, then $\mathrm{cov.gen}(X,Y)=\mathrm{cov.gen}(X)=\mathrm{cov.gen}(Y)$. The converse does not hold: for very general $K3$ surfaces $S_1,S_2$ one has $\mathrm{cov.gen}(S_1,S_2)=\mathrm{cov.gen}(S_1)=\mathrm{cov.gen}(S_2)=1$ (see \cite[Proposition~1.14]{RobOli}).

The results of this note compute $\mathrm{cov.gen}(X,Y)$ for further pairs of $K$-trivial varieties. In the case of a $K3$ surface and an abelian surface we obtain:

\begin{propIntr}\label{SA}
Let $S$ and $A$ be a very general $K3$ surface and a very general abelian surface, respectively. Then
\[
\covgen(S,A)=3.
\]
\end{propIntr}
For two simple abelian surfaces $A_1$ and $A_2$, the invariant $\mathrm{cov.gen}(A_1,A_2)$ coincides with the minimal geometric genus of a non-trivial curve $C\subset A_1\times A_2$, where non-trivial means that both projections $C\to A_1$ and $C\to A_2$ are non-constant. Indeed, given such a curve $C$, the surface $C^{(2)}$ maps to $A_1\times A_2$ and yields a correspondence covered by curves of genus $g(C)$. Conversely, if $Z\subset A_1\times A_2$ is a correspondence dominating both $A_1$ and $A_2$, then the general member of any covering family of curves on $Z$ must have non-constant projections to both factors: otherwise, if all members were contained in the fibers of one of the projections, the image of $Z\to A_i$ would be one-dimensional.

Our main theorem can then be stated in the following form:

\begin{thmIntr}\label{A12}
Let $A_1,A_2$ be a very general pair of abelian surfaces. Then
\[
\covgen(A_1,A_2)=6.
\]
\end{thmIntr}
It is straightforward to obtain the upper bound $\covgen(A_1,A_2)\le 6$. Indeed, choose genus $2$ curves $C_1\subset A_1$ and $C_2\subset A_2$, and consider their hyperelliptic maps $f_i:C_i\to \mathbb P^1$ given by the $g^1_2$. Form the fiber product $C:=C_1\times_{\mathbb P^1} C_2$ after imposing that three of the six branch points coincide. A direct computation shows that the resulting curve has genus $6$, and it maps non-trivially to both $A_1$ and $A_2$.

The lower bound is more subtle and follows a classical, variational argument. The key point is a careful analysis of the differential of the Torelli map along a family of genus $5$ curves whose Jacobians split (up to isogeny) as a product of two abelian surfaces and an elliptic curve. Controlling the rank of the relevant multiplication maps---and describing the geometric consequences when this rank drops---shows that the corresponding locus has dimension at most $5$. It follows that, for very general $A_1$ and $A_2$, one must have $\covgen(A_1,A_2)\ge 6$. 

In other words, we obtain the following intermediate statement, which may be of independent interest.

\begin{thmIntr}
The Torelli locus in $\mathcal A_5$ intersects $\mathcal A_2\times \mathcal A_2\times \mathcal A_1$ and its Hecke translates along subschemes of dimension at most $5$.
\end{thmIntr}

The expected dimension of such an intersection is $4$, but it is not difficult to see that $5$-dimensional components do occur.
For instance, the genus $6$ curve constructed above inside the product of two very general abelian surfaces $A_1\times A_2$ degenerates to a genus $5$ curve when $(A_1,A_2)$ varies in suitable codimension-$1$ loci of $\mathcal A_2\times \mathcal A_2$.

\textbf{Acknowledgements.} The problem studied in this paper was suggested to us by Rob Lazarsfeld. We are grateful to Rob Lazarsfeld and Olivier Martin for several illuminating and helpful discussions on related questions.
\section{Covering genus of a very general abelian surface and a $K3$ surface}

We begin by constructing a family of genus $3$ curves dominating a very general $K3$ surface $S$ and a very general abelian surface $A$. The reader may compare the construction below with \cite[Proposition~1.14]{RobOli}, as they are similar in nature.

\begin{ex}
Let $S$ and $A$ be a very general $K3$ surface and a very general abelian surface, respectively. Then
\[
\covgen(S,A)\le 3.
\]

We now construct such a family. A very general $(1,2)$-polarized abelian surface carries a non-isotrivial pencil of genus $3$ curves induced by its polarization. The pencil is non-isotrivial because it contains a nodal member of geometric genus $2$ (for instance, arising from a degree-$2$ isogeny from a principally polarized abelian surface). Moreover, any abelian surface admits an isogeny from a $(1,2)$-polarized abelian surface. Therefore, for a very general abelian surface $A$, we can construct a family $\mathcal C \to \mathbb P^1$ of genus $3$ curves together with a map $\mathcal C \to A$.

For general $t$, the curve $C_t$ admits a non-trivial map onto an elliptic curve
\[
E_t=\mathrm{Jac}(C_t)/A.
\]
By Torelli, the elliptic curve $E_t$ varies in moduli: indeed, $\mathrm{Jac}(C_t)$ is isogenous to $A\times E_t$, and the curves $C_t$ vary in moduli.

Now let $\mathcal E \to B$ be a non-isotrivial family of elliptic curves covering $S$; such a family exists by \cite[Corollary, p.~2]{gounelas}. Using relative Hilbert schemes, we obtain countably many quasi-projective varieties $\mathcal B_k$ parametrizing triples $(t,b,f_b)$, where $t\in \mathbb P^1$, $b\in B$, and $f_b:C_t \to E_b$ is a morphism.
Since the map $\bigcup_k \mathcal B_k\to \mathbb P^1 \times B$ is surjective (both families of elliptic curves vary maximally in moduli), we can choose an irreducible component $\mathcal B_0$ dominating both $\mathbb P^1$ and $B$. The corresponding universal family $\mathcal X\to \mathcal B_0$ then admits a natural map to $A\times S$ whose image dominates both factors, and by construction it is covered by curves of genus $3$.
\end{ex}

We devote this section to the proof of \autoref{SA}.

\begin{lem}\label{stime}\label{prelemma}
Let $S$ be either an abelian surface or a $K3$ surface, and let $\mathcal C\to B$ be a family of curves on $S$ such that the induced map $B\to \mathrm{Hilb}_S$ is finite onto its image (with $B$ reduced and irreducible). For $b\in B$ general, let $\widetilde C_b$ be the normalization of $C_b$. Then
\[
g(\widetilde C_b)\ge \dim(B).
\]
Moreover, if for general $b$ the map $\pi_b:\widetilde C_b\to S$ is not a local immersion, then
\[
g(\widetilde C_b)\ge \dim(B)+1.
\]
\end{lem}

\begin{proof}
We give a sketch using a simple idea from \cite{families}; for a different proof see \cite[Chap.~3, Sec.~B]{def}. There is a canonical contraction map
\[
T_bB\otimes H^0(\omega_S)\to H^0(\omega_{C_b}),
\]
obtained by pulling back a holomorphic $2$-form on $S$ to a normalization of the family and contracting against a normal direction (note that $T_bB=H^0(\mathcal N_{C_b|\mathcal C})$; see \cite[Set-up~2.4]{families} for details). For $b$ general, this map is injective and factors through
\[
H^0\!\left(\omega_{\widetilde C_b}(-Z(d\pi_b))\right)\hookrightarrow H^0(\omega_{\widetilde C_b}),
\]
where $Z(d\pi_b)$ denotes the ramification divisor of $\pi_b:\widetilde C_b\to S$. The claims follow.
\end{proof}
We can now prove \autoref{SA}.
\begin{proof}
A very general abelian surface is simple, hence it contains no elliptic curves mapping nontrivially to it. Therefore $\covgen(A,S)>1$. It remains to show that $\covgen(S,A)>2$. We argue by contradiction.

First observe that if $\mathcal C \to B$ is a family of genus $2$ curves equipped with dominant rational maps
\[
\varphi: \mathcal C \dashrightarrow A
\qquad\text{and}\qquad
\psi: \mathcal C \dashrightarrow S,
\]
then $\mathcal C \to B$ must be isotrivial (this follows by Torelli since $\mathrm{Jac}( C_b)$ is isogenous to $A$). Moreover, for  general $b\in B$ the map $\varphi_b:C_b \to A$ is birational onto its images, since a very general abelian surface contains neither rational nor elliptic curves. The same holds for $\psi_b:C_b\to S$: indeed, $C_b$ cannot admit a non-constant map to an elliptic curve (otherwise $\mathrm{Jac}(C_b)$, and hence $A$, would contain an elliptic curve), and $S$ cannot be covered by rational curves.

Consider
\[
M_g^{\mathrm{bir}}(S):=\{[C] \in \mathcal M_g \mid \exists\, f: C \to S \text{ birational onto its image}\}.
\]
Then the locus
\[
\{[A] \in \mathcal A_{(1,d)} \mid \exists\, [C]\in \mathcal M_2 \text{ admitting maps } C \to A,\ C \to S \text{ birational onto their images}\}
\]
coincides with
\[
\{[A]\in \mathcal A_{(1,d)} \mid A \text{ is isogenous to some } [A']\in T_2(M_2^{\mathrm{bir}}(S))\},
\]
where $T_2: \mathcal M_2 \to \mathcal A_{(1,1)}$ is the Torelli map.

This last locus is contained in a countable union of proper Zariski-closed subsets. Indeed, by \autoref{stime}, the locus $M_2^{\mathrm{bir}}(S)$ is contained in a countable union of Zariski-closed subsets of dimension $\le 2$, hence proper. This finishes the proof.
\end{proof}

\section{Genus of curves on distinct abelian surfaces}
We divide the proof of \autoref{A12} into several steps. Since the statement is invariant under isogeny, we may assume without loss of generality that the abelian surfaces are principally polarized. We denote by $\mathcal A_g$ the moduli space of principally polarized abelian varieties of dimension $g$.

Before turning to the proof, we record the following elementary observation.

\begin{lem}\label{lem:surj-jac}
Suppose that $A_1$ and $A_2$ are simple and not isogenous. Then, for any curve $C\to A_1\times A_2$ with non-trivial projections to both factors, the induced morphism
\[
\mathrm{Jac}(C)\longrightarrow A_1\times A_2
\]
is surjective. In particular,
\[
H^0(C,\omega_C)=H^0(\Omega^1_{A_1})\oplus H^0(\Omega^1_{A_2})\oplus H^0\!\left(\Omega^1_{\ker(\mathrm{Jac}(C)\to A_1\times A_2)}\right).
\]
\end{lem}

\begin{proof}
Let $B\subset A_1\times A_2$ be the image of $\mathrm{Jac}(C)\to A_1\times A_2$. Then $B$ is an abelian subvariety. We argue that $\dim B=4$.

Assume by contradiction that $\dim B<4$ and consider the possible values.

\begin{itemize}
\item If $\dim B=1$, then the image of $B$ in at least one factor $A_i$ is an elliptic curve, contradicting the simplicity of $A_i$.

\item If $\dim B=2$, then both projections $B\to A_1$ and $B\to A_2$ are non-trivial, hence surjective by simplicity. It follows that $A_1$ and $A_2$ are both isogenous to $B$, and therefore isogenous to each other, a contradiction.

\item If $\dim B=3$, then the quotient $(A_1\times A_2)/B$ is an elliptic curve, so $A_1\times A_2$ is isogenous to $B\times E$ for some elliptic curve $E$. By Poincar\'e complete reducibility, $E$ appears as an isogeny factor of $A_1\times A_2$, hence of $A_1$ or $A_2$, again contradicting simplicity.
\end{itemize}

Therefore $\dim B=4$ and $\mathrm{Jac}(C)\to A_1\times A_2$ is surjective. The decomposition of $H^0(C,\omega_C)=H^0(\Omega^1_{\mathrm{Jac}(C)})$ follows.
\end{proof}

An immediate consequence is the following.

\begin{cor}\label{cor:iso}
If $\covgen(A_1,A_2)=g$ for very general $A_1,A_2\in \mathcal A_2$, then there exists a $6$-dimensional subvariety $B\subset \mathcal M_g$ whose general point $[C]\in B$ satisfies that $\mathrm{Jac}(C)$ contains the product of two abelian surfaces as an isogeny factor.
\end{cor}

\begin{proof}
This follows from \autoref{lem:surj-jac} and the fact that
\[
\dim(\mathcal A_2\times \mathcal A_2)=\dim(\mathcal A_2)+\dim(\mathcal A_2)=3+3=6.
\]
\end{proof}

\subsection{Curves of genus $\le 4$ on the product of two abelian surfaces}

We begin by ruling out the existence of curves of genus $\le 4$ mapping nontrivially to the product of two very general abelian surfaces.

\begin{prop}\label{prop:covgen-ge5}
Under the hypotheses of \autoref{A12}, one has $\covgen(A_1,A_2)\ge 5$.
\end{prop}

\begin{proof}
The cases $g\le 3$ follow, for instance, from \autoref{cor:iso}.

We now treat genus $4$. Let $\mathcal H\subset \mathcal M_4$ be the locus parametrizing genus $4$ curves that admit a non-constant map to $A_2$; it is a countable union of closed subsets. Let $B\subseteq \mathcal H$ be an irreducible component, and let
\[
p:W\longrightarrow B
\]
be the induced family, endowed with a dominant map $W\to A_2$ restricting fiberwise to maps $W_b\to A_2$.

If the map $W_b\to A_2$ is birational onto its image, then by \autoref{prelemma} we have $\dim(B)\le 4$.
If $W_b\to A_2$ is not birational onto its image, then its image must be a genus $2$ curve $C_2\subset A_2$, and the induced cover
$W_b\to C_2$ can ramify over at most two points; in particular, again $\dim(B)\le 4$ (two parameters for the ramification points and two parameters for translating $C_2$ inside $A_2$).

On the other hand, the fibers of the natural map $B\to \mathcal M_4$ have dimension at least $2$, since one can postcompose a given map $W_b\to A_2$ with translations of $A_2$. It follows that the locus of abelian surfaces $A_1\in \mathcal A_2$ admitting a nontrivial map $W_b\to A_1$ for some $b\in B$ has dimension at most $2$, hence is a proper subset of $\mathcal A_2$.

Therefore, fixing $A_2$ and choosing $A_1$ very general, there is no genus $4$ curve mapping nontrivially to $A_1\times A_2$. This shows that $\covgen(A_1,A_2)\ge 5$.
\end{proof}

To rule out $\mathrm{cov.gen}(A_1,A_2)=5$ for very general $A_1,A_2$, we borrow some ideas from \cite{cvt,cv}. We study the dimension of families of genus $5$ curves whose Jacobian is isogenous to a product of two abelian surfaces and an elliptic curve (possibly varying in moduli). Let $\mathcal A_5$ be the moduli space of principally polarized abelian fivefolds, and let $\mathcal M_5$ be the moduli space of genus $5$ curves, with Torelli map
\[
T_5:\mathcal M_5\to \mathcal A_5.
\]
We would like to bound the dimension of certain subvarieties of $\mathcal M_5$ by controlling their tangent spaces. Since $\mathcal M_5$ and $\mathcal A_5$ are stacks (and their coarse moduli spaces may be singular), one should work on suitable covers (for instance by imposing a level structure) in order to speak unambiguously about tangent spaces. To simplify notation, we will not make this explicit.

It is well known that if $A\in \mathcal A_5$, then there is a canonical identification
\[
T_{[A]}\mathcal A_5 \cong \Sym^2 T_{A,0},
\]
see \cite{cv}, p.~553. Let $A\in \mathcal A_5$ be isogenous to a product $A_1\times A_2\times E$, where $A_1,A_2$ are principally polarized abelian surfaces and $E$ is an elliptic curve (we write $A \approx A_1\times A_2\times E$).

\begin{prop}\label{deform}
The vector subspace of the deformation space $\Sym^2T_{A,0}$ parametrizing deformations of triples $(A,L,A\twoheadrightarrow A_1\times A_2)$, where $L$ is a polarization on $A$, is
\[
V:=\Sym^2T_{A_1,0} \oplus \Sym^2 T_{A_2,0} \oplus \Sym^2 T_{E,0}.
\]
\end{prop}

If $[C]\in \mathcal M_5$, then a neighborhood of $[C]$ in $\mathcal M_5$ can be identified with a neighborhood of $0$ in $H^0(C,K_C^{\otimes 2})^*/\mathrm{Aut}(C)$. The Torelli map induces a morphism on cotangent spaces
\[
dT_5^*: T_{[JC]}^* \mathcal A_5 \to T_{[C]}^*\mathcal M_5,
\]
which, under the standard identifications, is the multiplication map
\[
\mathrm{ev}: \Sym^2 H^0(K_C) \to H^0(K_C^{\otimes 2}).
\]

Let
\[
M:= \Bigl\{ [C]\in \mathcal M_5 \ \Big|\ JC \approx A_1\times A_2 \times E \text{ for some } A_1,A_2 \in \mathcal A_2,\ E\in \mathcal A_1 \Bigr\}\subseteq \mathcal M_5,
\]
and set
\[
W:= \bigl(T_{A_1,0} \otimes T_{A_2,0}\bigr)\ \oplus\ \bigl(T_{A_1,0}\oplus T_{A_2,0}\bigr)\otimes T_{E,0},
\]
so that
\[
T_{[JC]}\mathcal A_5=\Sym^2 T_{A,0}=V\oplus W.
\]

\begin{prop}
Assume the general point $[C]$ of $M$ is non-hyperelliptic. Then the cotangent space $T_{[C]}^*M$ is a quotient of $T_{[C]}^*\mathcal M_5/\mathrm{ev}(W)$.
\end{prop}

\begin{proof}
The differential of the Torelli map sends $T_{[C]}M$ into $V\subset T_{[JC]}\mathcal A_5$. Since $C$ is non-hyperelliptic, the Torelli map is an immersion in a neighborhood of $[C]$. Using the identifications $T_{[C]}^*\mathcal M_5\cong H^0(K_C^{\otimes 2})$ and $T_{[JC]}^*\mathcal A_5 \cong \Sym^2 H^0(K_C)$, we obtain a commutative diagram
\[
\xymatrix{
T_{[C]}^*M & V \ar@{->>}[l] \\
T_{[C]}^*\mathcal M_5 \ar@{->>}[u] & \Sym^2 H^0(K_C) \ar@{->>}[l]^{\mathrm{ev}} \ar@{->>}[u] & V\oplus W \ar[ul]_{\mathrm{pr}_{V}} \ar[l]_{\cong}
}
\]
In particular, the composition $W \hookrightarrow V\oplus W \xrightarrow{\mathrm{pr}_V} V \twoheadrightarrow T_{[C]}^*M$ is zero, hence $\mathrm{ev}(W)$ maps to zero in $T_{[C]}^*M$. Therefore the quotient $T_{[C]}^*\mathcal M_5/\mathrm{ev}(W)$ surjects onto $T_{[C]}^*M$, as claimed.
\end{proof}
We now prove \autoref{A12}. Given a genus $5$ curve $C$ together with an isogeny $\mathrm{Jac}(C)\simeq A_1\times A_2 \times E$ (hence an identification
$H^0(\omega_C)=H^0(\Omega^1_{A_1})\oplus H^0(\Omega^1_{A_2})\oplus H^0(\omega_E)$), let $m$ denote the \emph{mixed} multiplication map
\[
m:\Bigl(H^0(\Omega^1_{A_1})\otimes H^0(\Omega^1_{A_2})\Bigr)\ \oplus\ \Bigl(H^0(\omega_E)\otimes H^0(\Omega^1_{A_1})\Bigr)\ \oplus\ \Bigl(H^0(\omega_E)\otimes H^0(\Omega^1_{A_2})\Bigr)\longrightarrow H^0(\omega_C^{\otimes 2}).
\]
As a direct corollary of the discussion above, we obtain the following.

\begin{cor}\label{rankmb}
Suppose $\widetilde B\subset \mathcal M_5$ is a family of curves such that $\mathrm{Jac}(C_b)\simeq A_1(b)\times A_2(b)\times E_b$. Then
\[
\dim(\widetilde B) \le 12 - \mathrm{rank}(m_b).
\]
\end{cor}

Thus \autoref{A12} is reduced to proving that $\mathrm{rank}(m_b)\ge 7$. For simplicity, we omit the subscript $b$ in the notation; we will reintroduce it when needed.

\subsection{The rank of the multiplication map}
We begin by fixing the notation that we will use throughout the section.

\begin{setup}\label{setup45}
Suppose, for contradiction, that for all $[A_1],[A_2] \in \mathcal A_2$ there exists a genus $5$ curve mapping non-trivially onto $A_1\times A_2$ (with both projections non-constant). By a standard Hilbert-scheme argument, we can construct a family of genus $5$ curves $\widetilde{\mathcal C} \to \widetilde B$, where $\widetilde B$ is a quasi-projective variety endowed with a dominant map $\widetilde B \to \mathcal A_2\times \mathcal A_2$, together with a morphism $\widetilde \pi: \widetilde{\mathcal C} \to \mathcal A \times \mathcal A$ between the universal families (realizing the general curve in the product of two very general abelian surfaces).

After the base change
\[
B=\widetilde B\times _{\mathcal A_2\times \mathcal A_2}(\mathcal A\times \mathcal A),
\]
we obtain an induced universal family $\mathcal C\to B$ and a morphism $\pi: \mathcal C \to \mathcal A \times \mathcal A$ over $\mathcal A_2 \times \mathcal A_2$, defined on points by
\[
(p,a_1,a_2)\longmapsto (p+a_1,p+a_2).
\]

For instance, if ${\pi_1}_{|_{C}}:C\to A_1$ is birational onto its image, then the fibers of $B_{A_1}\to \mathrm{Hilb}_{A_1}$ are $2$-dimensional and contain the translates of the curve in $A_2$\footnote{Here we use the standard fact that the only deformations of a curve on an abelian variety that are trivial in moduli are induced by translations.}. Note that $\dim(B)=10$ and $\dim(B_{A_i})=7$. We basically enlarged the original space by adding translations in $A_1$ and $A_2$.

Under these hypotheses, we will prove that $\mathrm{rank}(m)\ge 7$ for a general curve $C=C_b$ over $B$, contradicting \autoref{rankmb}.
\end{setup}

\subsection{The $H^0(\Omega^1_{A_1})\otimes H^0(\Omega^1_{A_2})$ part}
The goal of this subsection is to rule out the case in which
\[
m_{1,2}: H^0(\Omega^1_{A_1})\otimes H^0(\Omega^1_{A_2})\to H^0(\omega_C^{\otimes 2})
\]
has rank $3$ (a smaller rank is impossible by Hopf's lemma). Each pencil $|H^0(\Omega^1_{A_i})|$ induces a map $\varphi_i :C\to \mathbb P^1$, and via the Segre embedding we obtain a map
\[
C \to \mathbb P^1\times \mathbb P^1\subset \mathbb P^3.
\]
If $\mathrm{rank}(m_{1,2})=3$, then $C$ is degenerate in $\mathbb P^3$, hence the two maps $\varphi_i$ define the same linear series. Up to identifying the targets $\mathbb P^1$, we may therefore assume that $\varphi:=\varphi_1=\varphi_2$.

Since $H^0(\Omega^1_{A_1})\cap H^0(\Omega^1_{A_2})=0$ inside $H^0(\omega_C)$, this can only happen if the associated linear series have fixed parts. Thus we can write
\[
|H^0(\Omega^1_{A_i})|=D_i+|\varphi^{-1}(t)|,
\]
where $D_1,D_2\neq 0$ are the fixed divisors. If we denote by $\pi_i:C\to A_i$ the projections, then $D_i=Z(d\pi_i)$.

Now observe that ${\pi_i}_{|_C}$ is not birational onto its image for $i=1,2$. Indeed, assume by contradiction that $\pi_2$ is birational. In this situation, $\pi$ induces a map $B_{A_2} \dashrightarrow \mathrm{Hilb}_{A_2}$ with $2$-dimensional fibers (induced by translations in $A_1$). We then reach a contradiction with \autoref{prelemma}, since we would obtain
\[
g(C)>\dim\bigl(\mathrm{Im}(B_{A_2}\to \mathrm{Hilb}_{A_2})\bigr)-2=5,
\]
where the inequality is strict because $Z(d\pi_i)\neq 0$.

For future reference, we record this in the following statement.

\begin{lem}\label{fixedpart}
In the hypotheses of \autoref{setup45}, for a general curve $C$ in the family, if the linear series $|H^0(\Omega^1_{A_i})|$ has a fixed divisor $D_i>0$, then ${\pi_i}_{|_C}$ is not birational onto its image. Moreover, $D_i=Z(d{\pi_i}_{|_C})$.
\end{lem}

We now consider the case in which both $\pi_1$ and $\pi_2$ are not birational onto their images $\pi_i(C)$. Let $C_i$ denote the normalization of $\pi_i(C)$. The linear series $|H^0(\Omega^1_{A_i})|$ is pulled back from the linear series induced on $C_i$. In other words, we have a commutative diagram
\[
\begin{tikzcd}
    & C\arrow{dl}{\pi_1} \arrow{dd}{\varphi} \arrow{dr}{\pi_2} & \\
    C_1 \arrow{dr}{\psi_1} & & C_2 \arrow{dl}{\psi_2} \\
    & \mathbb P^1 &
\end{tikzcd}
\]
with $\psi_i$ the morphism induced by $|H^0(\Omega^1_{A_i})|$ on $C_i$, after identifying the targets $\mathbb P^1$ appropriately. We now claim that the geometric genera of $C_1$ and $C_2$ are both equal to $2$.

Suppose by contradiction that $g(C_2)=3$ (higher genus is excluded by the Hurwitz formula). Since $g(C)=5$, the map $C\to C_2$ must then be \'etale. In particular, if we fix $A_2$, the number of moduli of $C$ is the same as the number of moduli of $C_2\in \mathrm{Hilb}_{A_2}$. Equivalently, the rational map $B_{A_2}\dashrightarrow \mathrm{Hilb}_{A_2}$ has $2$-dimensional fibers\footnote{The fibers are parametrized by translations in $A_1$, hence they are isotrivial families.}. By \autoref{prelemma}, any family of genus $3$ curves in $\mathrm{Hilb}_{A_2}$ has dimension at most $3$. Therefore
\[
\dim\bigl(\mathrm{Im}(B_{A_2}\to \mathrm{Hilb}_{A_2})\bigr)\le 3,
\]
which contradicts $\dim(B_{A_2})=7$. We conclude that $g(C_i)=2$ for $i=1,2$, and that the maps $\psi_i$ are the $g^1_2$'s (since they must be contained in the canonical linear series of $C_i$).

Now consider the space $\mathcal M_2^W$ parametrizing tuples $(C,p_1,p_2,p_3,p_4,p_5,p_6)$ where $C$ is a genus $2$ curve and the $p_i$'s are ordered Weierstrass points. It is a finite cover of $\mathcal M_2$. There is a finite dominant map $\mathcal M_2^W\to \mathcal M_{0,6}$ sending $(C,p_1,\dots,p_6)$ to $(f(p_1),\dots,f(p_6))$, where $f:C \to \mathbb P^1$ is the $g^1_2$.

For any pair $(C_1,C_2)\in \mathcal M_2^W\times \mathcal M_2^W$ we can consider the fiber product
\[
C=C_1\times_{\mathbb P^1} C_2 .
\]
This fiber product has geometric genus $9-k$, where
\[
k=\bigl|f_1(\{p_1,\dots,p_6\})\cap f_2(\{q_1,\dots,q_6\})\bigr|\in \{3,4,5,6\},
\]
and $\{p_1,\dots,p_6\}$ and $\{q_1,\dots,q_6\}$ are the Weierstrass points of $C_1$ and $C_2$, respectively. Note that $k=3$ for a general element of $\mathcal M_2^W \times \mathcal M_2^W$, since (after identifying the targets) we may assume
\[
f_1(p_1)=f_2(q_1)=0,\qquad f_1(p_2)=f_2(q_2)=1,\qquad f_1(p_3)=f_2(q_3)=\infty.
\]

Up to a finite base change (which we omit to simplify notation), we can consider the map
\[
B\to \mathcal M_2^W\times \mathcal M_2^W
\]
sending $C$ to $(C_1,C_2)$. By construction, the curve $C$ dominates the fiber product $C_1\times_{\mathbb P^1}C_2$. This forces $k\ge 4$ for the general curve in the image of $B$, hence, the map $B\to \mathcal M_2^W\times \mathcal M_2^W$ cannot be dominant. Therefore, its image is of dimension $<6$. On the other hand, the map $B\to \mathcal M_2^W\times \mathcal M_2^W$ has $4$-dimensional fibers (whose irreducible components are trivial families induced by translations in $A_1$ and $A_2$). This contradicts $\dim(B)=10$.

Summing up, we have proved the following.

\begin{lem}\label{m12}
The rank of $m_{1,2}$ is $4$.
\end{lem}

\subsection{The rank of $m$}
We are now ready to prove the main estimate.
\begin{lem}
The rank of $m$ is at least $7$.
\end{lem}

\begin{proof}
We have $\mathrm{Im}(m)=\mathrm{Im}(m_{1,2})+\mathrm{Im}(m_E)$, where $m_{1,2}$ is the map discussed in the previous subsection and
\[
m_E: H^0(\omega_E)\otimes \bigl(H^0(\Omega^1_{A_1})\oplus H^0(\Omega^1_{A_2})\bigr)\to H^0(K_C^{\otimes 2}).
\]
By \autoref{m12} we may assume that $\mathrm{rank}(m_{1,2})=4$, and clearly $\mathrm{rank}(m_E)=4$. Therefore, it suffices to prove that
\[
\dim\bigl(\mathrm{Im}(m_{1,2})\cap \mathrm{Im}(m_E)\bigr)\le 1.
\]
We split the argument into two cases.

\begin{itemize}
\item[\textbullet] \emph{Case 1: $\pi_i$ is birational for $i=1,2$.}
By \autoref{fixedpart}, the linear series $|H^0(\Omega^1_{A_i})|$ have no fixed part. Since $\mathrm{rank}(m_{1,2})=4$, the curve $C$ is non-degenerate in
\[
C\to \mathbb P^1 \times \mathbb P^1\subset \mathbb P^3
\]
and has degree $16$. Assume by contradiction that $\dim(\mathrm{Im}(m_{1,2})\cap \mathrm{Im}(m_E))=2$. Then we can find a line $\ell\subset \mathbb P^3$ cut out by divisors containing
\[
K_E=\sum_{i=1}^8 P_i.
\]
The line $\ell$ must intersect the quadric $\mathbb P^1\times \mathbb P^1$ properly; otherwise we would have $K_E\in |H^0(\Omega^1_{A_i})|$ for some $i$. Thus
\[
\ell \cdot (\mathbb P^1\times \mathbb P^1)=Q_1+Q_2.
\]
Since $\deg(K_E)=8$, the line $\ell$ meets the image of $C$ in $8$ points. As $C$ lies on $\mathbb P^1\times \mathbb P^1$, it follows that $C$ must be singular along $Q_1$ and $Q_2$. In particular, the points $P_1,\dots,P_8$ lie above $Q_1$ and $Q_2$ under the map $C\to \mathbb P^1\times \mathbb P^1$. Hence one of $Q_1,Q_2$ has at least $4$ preimages contained in $\mathrm{Supp}(K_E)$; without loss of generality, let
\[
F:=P_1+P_2+P_3+P_4.
\]
Then $\{P_1,P_2,P_3,P_4\}$ lies in the same fiber for both rulings $\varphi_i:C\to \mathbb P^1$. Therefore there exist divisors $D_i$ such that
\[
D_i+P_1+P_2+P_3+P_4\in |H^0(\Omega^1_{A_i})|.
\]
We deduce that
\[
H^0(\Omega^1_{A_1})(-F)\oplus H^0(\Omega^1_{A_2})(-F)\oplus H^0(K_E-F)\subset H^0(K_C-F).
\]
Since the three summands on the left-hand side are all non-zero, we obtain $h^0(K_C-F)\ge 3$. As $\deg(K_C-F)=4$, Clifford's theorem implies that $C$ is hyperelliptic. This contradicts the fact that $C$ is general: indeed, hyperelliptic curves are rigid in abelian varieties, whereas $C$ moves with $3$ moduli in both $A_1$ and $A_2$ (recall that $\dim(B_{A_i})=7$ and that the $\pi_i$'s are birational onto their images).

\item[\textbullet] \emph{Case 2: $\pi_2$ is not birational onto its image.}
Since $\dim(B_{A_2})=7$, we claim that $\deg(\pi_2)=2$. Suppose by contradiction that $\deg(\pi_2)=3$. Then $C_2=\pi_2(C)$ must have genus $2$ (it cannot have smaller genus since $A_2$ is simple, and it cannot have larger genus by the Hurwitz formula). Let $R$ be the ramification divisor of $C\to \widetilde{C_2}$; by Hurwitz, $\deg(R)=2$. We obtain
\[
\dim(B_{A_2})\le 4+\deg(R)=6,
\]
where $4$ is the dimension of the isotrivial part of the family induced by translations in $A_1$ and $A_2$, and we use that genus $2$ curves are rigid in abelian surfaces\footnote{Here we use implicitly the Riemann existence theorem: the curve $C$ can be reconstructed from the branch divisor together with the discrete monodromy data. In particular, the dimension of the family in moduli is bounded by the degree of the branch divisor, since $C_2$ is rigid in $A_2$ (up to translation).}. This contradicts $\dim(B_{A_2})=7$. Hence $\deg(\pi_2)=2$.

By a similar parameter count, we deduce that the geometric genus of $C_2=\pi_2(C)$ is $2$. This implies that the branch divisor $R_2$ has degree $4$. Thus the claim reduces to showing that the locus of (ramified) Prym curves in $\mathcal R_{2,4}$ whose Prym variety splits (up to isogeny) as $A_2\times E$ has codimension $2$ in $\mathcal R_{2,4}$ (which has dimension $7$). We refer the reader to \cite{Mumformprym} for the Prym construction and to \cite{ramifiedprym} for details on the ramified Prym.

Equivalently, it suffices to show that the image of the canonical map $B\to \mathcal R_{2,4}$ has dimension $\le 5$. The codifferential of the Prym map $\mathcal R_{2,4}\to \mathcal A_3$ can be identified with the multiplication map
\[
m_4:\mathrm{Sym}^2H^0(\omega_{C_2}\otimes \eta) \longrightarrow H^0(\omega_{C_2}^{\otimes 2}(R_2)),
\]
where $\eta$ is a suitable square root of $\mathcal O_{C_2}(R_2)$ (see, for instance, \cite[page 1151]{ramifiedprym}). In our situation, $m_4$ is injective with cokernel of rank $1$ everywhere. Moreover, for any $C_b$ we have a splitting
\[
H^0(\omega_{C_2}\otimes \eta)=H^0(\omega_E)\oplus H^0(\Omega^1_{A_1}).
\]
Arguing as in \autoref{rankmb}, we obtain
\[
\dim(B)\le \dim(\mathcal R_{2,4})-\mathrm{rank}\bigl(H^0(\omega_E)\otimes H^0(\Omega^1_{A_1})\to H^0(\omega_{C_2}^{\otimes 2}(R_2))\bigr)=5.
\]
This concludes the proof.
\end{itemize}
\end{proof}
\bibliography{refer}

@article {cvt,
    AUTHOR = {Ciliberto, Ciro and van der Geer, Gerard and Teixidor i Bigas, Monserrat},
     TITLE = {On the number of parameters of curves whose {J}acobians
              possess nontrivial endomorphisms},
   JOURNAL = {J. Algebraic Geom.},
  FJOURNAL = {Journal of Algebraic Geometry},
    VOLUME = {1},
      YEAR = {1992},
    NUMBER = {2},
     PAGES = {215--229},
      ISSN = {1056-3911,1534-7486},
   MRCLASS = {14H10 (14H40)},
  MRNUMBER = {1144437},
MRREVIEWER = {H.\ Lange},
}

@incollection {Mumformprym,
    AUTHOR = {Mumford, David},
     TITLE = {Prym varieties. {I}},
 BOOKTITLE = {Contributions to analysis (a collection of papers dedicated to
              {L}ipman {B}ers)},
     PAGES = {325--350},
 PUBLISHER = {Academic Press, New York-London},
      YEAR = {1974},
   MRCLASS = {14H40 (14K25 32G20)},
  MRNUMBER = {379510},
MRREVIEWER = {H.\ H.\ Martens},
}

@article {ramifiedprym,
    AUTHOR = {Marcucci, Valeria Ornella and Pirola, Gian Pietro},
     TITLE = {Generic {T}orelli theorem for {P}rym varieties of ramified
              coverings},
   JOURNAL = {Compos. Math.},
  FJOURNAL = {Compositio Mathematica},
    VOLUME = {148},
      YEAR = {2012},
    NUMBER = {4},
     PAGES = {1147--1170},
      ISSN = {0010-437X,1570-5846},
   MRCLASS = {14H40 (14C34)},
  MRNUMBER = {2956039},
MRREVIEWER = {Samuel\ Dalalyan},
       DOI = {10.1112/S0010437X12000280},
       URL = {https://doi.org/10.1112/S0010437X12000280},
}

@article {cv,
    AUTHOR = {Ciliberto, Ciro and van der Geer, Gerard},
     TITLE = {Subvarieties of the moduli space of curves parametrizing
              {J}acobians with nontrivial endomorphisms},
   JOURNAL = {Amer. J. Math.},
  FJOURNAL = {American Journal of Mathematics},
    VOLUME = {114},
      YEAR = {1992},
    NUMBER = {3},
     PAGES = {551--570},
      ISSN = {0002-9327,1080-6377},
   MRCLASS = {14H40 (14H10)},
  MRNUMBER = {1165353},
MRREVIEWER = {H.\ Lange},
       DOI = {10.2307/2374769},
       URL = {https://doi.org/10.2307/2374769},
}

@article {gounelas,
    AUTHOR = {Chen, Xi and Gounelas, Frank},
     TITLE = {Curves of maximal moduli on {K}3 surfaces},
   JOURNAL = {Forum Math. Sigma},
  FJOURNAL = {Forum of Mathematics. Sigma},
    VOLUME = {10},
      YEAR = {2022},
     PAGES = {Paper No. e36, 21},
      ISSN = {2050-5094},
   MRCLASS = {14J28 (14N35)},
  MRNUMBER = {4436594},
MRREVIEWER = {G.\ K.\ Sankaran},
       DOI = {10.1017/fms.2022.24},
       URL = {https://doi.org/10.1017/fms.2022.24},
}

@book {def,
    AUTHOR = {Harris, Joe and Morrison, Ian},
     TITLE = {Moduli of curves},
    SERIES = {Graduate Texts in Mathematics},
    VOLUME = {187},
 PUBLISHER = {Springer-Verlag, New York},
      YEAR = {1998},
     PAGES = {xiv+366},
      ISBN = {0-387-98438-0; 0-387-98429-1},
   MRCLASS = {14H10 (14-02 14D20 14D22)},
  MRNUMBER = {1631825},
MRREVIEWER = {R.\ F.\ Lax},
}

@article {RobOli,
    AUTHOR = {Lazarsfeld, Robert and Martin, Olivier},
     TITLE = {Measures of association between algebraic varieties},
   JOURNAL = {Selecta Math. (N.S.)},
  FJOURNAL = {Selecta Mathematica. New Series},
    VOLUME = {29},
      YEAR = {2023},
    NUMBER = {3},
     PAGES = {Paper No. 46, 37},
      ISSN = {1022-1824,1420-9020},
   MRCLASS = {14N05 (14E05 14E20 14J70)},
  MRNUMBER = {4602048},
       DOI = {10.1007/s00029-023-00849-8},
       URL = {https://doi.org/10.1007/s00029-023-00849-8},
}

@misc{families,
      title={On the locus of curves with a map to a fixed variety}, 
      author={Yeuk Hay Joshua Lam and Federico Moretti and Giovanni Passeri},
      year={2023},
      eprint={2312.16974},
      archivePrefix={arXiv},
      primaryClass={math.AG}
}

@inproceedings{pietro1995generic,
  title={generic abelian varieties},
  author={Pirola, Gian Pietro},
  booktitle={Abelian Varieties: Proceedings of the International Conference, Held in Egloffstein, Germany, October 3-8, 1993},
  pages={237},
  year={1995},
  organization={Walter de Gruyter}
}

@misc{engel2026matroidsintegralhodgeconjecture,
      title={Matroids and the integral Hodge conjecture for abelian varieties}, 
      author={Philip Engel and Olivier de Gaay Fortman and Stefan Schreieder},
      year={2026},
      eprint={2507.15704},
      archivePrefix={arXiv},
      primaryClass={math.AG},
      url={https://arxiv.org/abs/2507.15704}, 
}

@misc{engel2025optimalityprymtyurinconstructionmathcala6,
      title={Optimality of the Prym-Tyurin construction for $\mathcal{A}_6$}, 
      author={Philip Engel and Olivier de Gaay Fortman and Stefan Schreieder},
      year={2025},
      eprint={2512.04902},
      archivePrefix={arXiv},
      primaryClass={math.AG},
      url={https://arxiv.org/abs/2512.04902}, 
}

@article{AlexeevDonagiFarkasIzadiOrtega+2020+163+217,
url = {https://doi.org/10.1515/crelle-2018-0005},
title = {The uniformization of the moduli space of principally polarized abelian 6-folds},
title = {},
author = {Valery Alexeev and Ron Donagi and Gavril Farkas and Elham Izadi and Angela Ortega},
pages = {163--217},
volume = {2020},
number = {761},
journal = {Journal für die reine und angewandte Mathematik (Crelles Journal)},
doi = {doi:10.1515/crelle-2018-0005},
year = {2020},
lastchecked = {2026-04-02}
}
\bibliographystyle{alphaspecial}

\end{document}